\documentclass[11pt]{amsart}
\usepackage{setspace}
\usepackage{amsmath,amssymb,mathrsfs,color}
\usepackage{hyperref}
\hypersetup{hypertex=true,
	colorlinks=true,
	linkcolor=blue,
	anchorcolor=blue,
	citecolor=blue}
\allowdisplaybreaks

\def\d{{\rm{d}}}

\def\N{{\mathbb N}}

\def\R{{\mathbb R}}

\newtheorem{lemma}{Lemma}[section]
\newtheorem{theorem}[lemma]{Theorem}
\newtheorem{remark}[lemma]{Remark}

\allowdisplaybreaks
\numberwithin{equation}{section}

\begin{document}

\title[Small perturbations may change the sign of Lyapunov exponents for SDEs]{Small perturbations may change the sign of Lyapunov exponents for linear SDEs}

\author{Xianjin Cheng}
\address{X. Cheng: School of Mathematics,
	Dalian University of Technology, Dalian 116024, P. R. China}
\email{xjcheng1119@hotmail.com; xjcheng@mail.dlut.edu.cn}

\author{Zhenxin Liu}
\address{Z. Liu: School of Mathematical Sciences and State Key Laboratory of Structural Analysis for
Industrial Equipment, Dalian University of Technology, Dalian 116024, P. R. China}
\email{zxliu@dlut.edu.cn}

\author{Lixin Zhang}
\address{L. Zhang (Corresponding author): School of Mathematical Sciences,
	Dalian University of Technology, Dalian 116024, P. R. China}
\email{lixinzhang8@163.com; 1031476306@mail.dlut.edu.cn}

\date{January 3, 2023}

\begin{abstract}
  In this paper, we study the existence of $n$-dimensional linear stochastic differential equations (SDEs) such that the sign of Lyapunov exponents is changed under an exponentially decaying perturbation. First, we show that the equation with all positive Lyapunov exponents will have $n-1$ linearly independent solutions with negative Lyapunov exponents under the perturbation. Meanwhile, we prove that the equation with all negative Lyapunov exponents will also have solutions with positive Lyapunov exponents under another similar perturbation.  Finally, we show that other three kinds of perturbations which appear at different positions of the equation will change the sign of Lyapunov exponents.  \\
\end{abstract}

\maketitle

\centerline{\em Dedicated to Professor Jinqiao Duan for his 60th birthday}
\medskip

\section{Introduction}
Lyapunov exponent measures the rate of separation between the  orbit of a point and those of nearby points. It directly reflects the asymptotic behavior of systems.  Therefore, discussing how the sign of Lyapunov exponents is changed plays an important role in studying dynamical systems. Barreira and Pesin \cite{Pesin1977Lyapunov} mentioned that the zero solution remains exponentially stable for the nonlinear equation under an additional condition called Lyapunov regularity. Similarly, Bochi and Viana \cite{2004Lyapunov} proved the continuity of Lyapunov exponents with respect to dynamical systems defined on manifolds. They both state that small perturbations cannot change the sign of Lyapunov exponents. Then when can the sign of Lyapunov exponents be changed?
\par
First, let us consider the following $n$-dimensional ordinary differential equation
\begin{equation}\label{equ.19}
\d X=A(t)X\d t,
\end{equation}
where $A:\R\rightarrow \R^{n\times n}$ is a bounded $C^{\infty}$ function. It has the fundamental solution system $[X_{1},\dots,X_{n}]$. As we know, Lyapunov exponents of equation (\ref{equ.19}) are defined by the formula
\begin{equation}
\varlimsup\limits_{t\rightarrow\infty}\frac{1}{t}\ln\Vert X_{i}(t)\Vert=	\lambda_{i}, \quad  i=1,\ldots,n,
\end{equation}
 where $\Vert v\Vert=\max_{1\leq i\leq n}|v_i|$, for $v=(v_1,...,v_n)\in \R^n$. The limitation with respect to every solution exists according to Oseledets \cite{1968A} and the equation only has $n$ (counted according to multiplicity) Lyapunov exponents. Recently,
 Izobov \cite{2022On}  proved the existence of an equation with all positive Lyapunov exponents and an exponentially decaying perturbation $Q$ satisfying the estimate $$\Vert Q(t)\Vert \leq Ce^{-\sigma t} \quad  \text{for}~t\geq 0,$$
 where $C$ and $\sigma$ are positive constants such that the perturbed equation
 $$ \d X=[A(t)+Q(t)]X\d t$$
 has $n-1$ linearly independent  solutions with negative Lyapunov exponents. This is an anti-Perron \cite{1930Die} version --- a version opposite to the well-known Perron effect of changing the negative Lyapunov exponents of  equation (\ref{equ.19}) to positive ones. Furthermore, Leonov \cite{2014Lyapunov} and Izobov \cite{2012Lyapunov} also gave many discussions about how the sign of Lyapunov exponents is changed under the perturbation.

In this paper, we want to study It\^{o} linear stochastic differential equation
\begin{equation}\label{equ.1}
	\d X=A(t)X\d t+F(t)X\d W,
\end{equation}
where $A:\R\rightarrow \R^{n\times n}$ and $F:\R\rightarrow \R^{n\times n}$ are bounded $C^{\infty}$ coefficients and $W$ is a one dimensional Brownian motion. According to \cite{2001Lyapunov}, for almost all $\omega$, there exists a $d(\omega)\in \N_{+}$ such that for $1,\ldots,d(\omega),$ there exists a sequence of linear subspaces $\left\{V_{i}(\omega)\subset\R^{n}| i=1,\ldots,d(\omega)\right\}$ satisfying
\begin{center}
$V_{d(\omega)}(\omega)\subset\ldots\subset V_{i}(\omega)\subset\ldots\subset V_{1}(\omega).$
\end{center}
Then for any $v\in V_{i}(\omega)\setminus V_{i+1}(\omega)$,
 Lyaounov exponent is denoted by the limit
\begin{equation}
\lambda_{i}(v,\omega):=	\varlimsup\limits_{t\rightarrow\infty}\frac{1}{t}\ln\Vert \Phi(t,\omega)v\Vert
\end{equation} where $\Phi(t,\omega)$ expresses the linear matrix--valued stochastic flow of equation (\ref{equ.1}) defined by Kunita \cite{1990Stochastic}.  So the  existence of Lyapunov exponents in this paper  always can be guaranteed. The perturbed equation is expressed as follows:
\begin{equation}\label{equ.2}
	\d X=[A(t)+Q(t)]X\d t+F(t)X\d W,
\end{equation}
with $C^{\infty}$ exponentially decaying  $n\times n$ perturbation matrices $Q(t)$ satisfying the
estimate
\begin{equation}\label{equ.3}
	\Vert Q(t)\Vert\leq Ce^{-\sigma t} \quad \text{for}~ t\geq0,
\end{equation}
where $C$ and $\sigma$ are positive constants.
\par We define the linearly independent property of solutions with deterministic initial conditions.  And we say that the solutions $X_{1},\ldots,X_{n}$ are {\em linearly independent}, if the initial conditions $v_{1},\ldots,v_{n}$  are linearly independent where $v_{i}\in \R^{n},$ $i=1,\ldots,n$. Although in this paper we say that many conclusions are  true for a unique solution $X$, they are also true for $cX, c\in\R\setminus\{0\}$.
\par
First of all, we will prove the existence of  equation (\ref{equ.1})  with all positive Lyapunov exponents  and perturbation $Q(t)$ such that perturbed equation (\ref{equ.2}) has $n-1$ linearly independent solutions with negative Lyapunov exponents. This paper shows that there exactly exist such equation and perturbation. And the perturbation is independent of $\omega$. This conclusion implies that although the  perturbation $Q(t)\rightarrow0$, $t\rightarrow\infty$, the sign of Lyapunov exponents is changed eventually.
\par Then,  assume that  original equation (\ref{equ.1}) has all negative Lyapunov exponents.  We can use the similar method  to get a perturbation such that   perturbed equation (\ref{equ.2}) has $n-1$ linearly independent solutions with positive Lyapunov exponents. This  is an ``anti-Perron'' example in the aspect of stochastic differential equations: there exists a perturbation such that the zero sulution loses its stability.
\par Finally, we show that other three kinds of perturbed equations whose perturbations appear at other positions also have the same conclusions above. They can be expressed by following equations
\begin{align}\label{equ.50}
\d X=[A(t)X+Q(t)]\d t+F(t)X\d W,
\end{align}
\begin{align}\label{equ.51}
\d X=A(t)X\d t+[F(t)+Q(t)]X\d W,
\end{align}
\begin{align}\label{equ.52}
\d X=A(t)X\d t+[F(t)X+Q(t)]\d W.
\end{align}
It is well-known that the Brownian motion satisfies the property:
$$\lim\limits_{t\rightarrow\infty}\dfrac{W(t)}{t}=0 \quad a.s.,$$
see e.g. \cite{2006Stochastic}. This fact plays an important role in the proof of theorems.  And it is worth noting that the perturbation in \eqref{equ.50}-\eqref{equ.52} is random and satisfies that for almost all $\omega$, there exists a $T(\omega)$ such that when $t>T(\omega)$,
$$
\Vert Q(t,\omega)\Vert\leq Ce^{-\sigma t}.	
$$
where $C$ and $\sigma$ are positive constants.
It reflects that the random perturbation also can take the same effects on the Lyapunov exponents of linear stochastic differential equations. And there also exists a perturbation which is at the inhomogeneous term of  equation (\ref{equ.1}) such that the signs of Lyapunov exponents have the same change. But it seems not feasible that the Lyapunov exponents changed from all positive to all negative.
\par The remainder of this paper is organized as follows. In the next
section, first of all, we prove the existence of  equation (\ref{equ.1}) with all positive Lyapunov exponents and  perturbation $Q$ such that the perturbed equation (\ref{equ.2}) has $n-1$ linearly independent solutions with negative Lyapunov exponents. Then, we show that there exist equation \eqref{equ.1} with all negative Lyapunov exponents and the perturbation $Q$ such that perturbed equation has $n-1$ linearly independent solutions with positive Lyapunov exponents. Finally, we prove that perturbed equation \eqref{equ.50} also has the similar conclusions above where the perturbation $Q(t,\omega)$ is  random. In the third section, we show that there exist an equation \eqref{equ.1} and the random perturbation $Q(t,\omega)$ such that the sign of Lyapunov exponents of equation \eqref{equ.51} is changed as same as above. And we illustrate perturbed equation \eqref{equ.52} has the same conclusions. In the last section, we discuss  other situations and give some complements to previous sections.

\section{Perturbations at drift term}
\par
In this section, we will discuss the effects of two different perturbations acting on the drift term of  equation \eqref{equ.1}, i.e. equations \eqref{equ.2} and \eqref{equ.50}. The following theorem gives the existence of  2-dimensional equation \eqref{equ.1} with all positive Lyapunov exponents  and $Q(t)$  such that  perturbed equation \eqref{equ.2} has a  solution with negative Lyapunov exponent. 

 \par   Since linear SDEs generate random dynamical systems (see Section 2.3 in \cite{1998Stochastic} for details), we can consider the solutions and other problems
from the pathwise viewpoint. We might as well suppose that the conclusions are true almost surely in this paper.

\subsection{2-dimensional situation}
\begin{theorem}\label{the.1}
	\begin{itemize}
      \item [(1)] For any constants $\lambda_{2}\geq\lambda_{1}>0$, $\theta>1$, there exists a $2$-dimensional equation \eqref{equ.1} with bounded $C^{\infty}$ coefficients and Lyapunov exponents $\lambda_{1},\lambda_{2}$.
      \item [(2)] For this equation and any constant $\sigma\in(0,\frac{(\theta-1)\lambda_{1}}{\theta})$, there exists a $C^{\infty}$ exponentially decaying perturbation $Q$ such that the perturbed equation \eqref{equ.2} has a unique solution $X_0$ with negative Lyapunov exponent $\lambda_{0}\leq\lambda(\theta,\lambda_{1},\sigma)<0$ where $\lambda(\theta,\lambda_{1},\sigma)$ will be given explicitly below;
       the Lyapunov exponent of the other  solution $X_{1}$ which is linearly independent of  $X_{0}$ is positive and is greater than $\vert\lambda_{0}\vert$.
	\end{itemize}
\end{theorem}
\begin{proof}

\par  \textbf{(1)} First of all, we will find an equation \eqref{equ.1} with Lyapunov exponents $\lambda_{1}$ and $\lambda_{2}$. We set $t_{k}=\theta^{k}, k\in \N$ and consider the linear diagonal equation
\begin{equation}\label{equ.4}
	\d X=\mathrm{diag}[a_{1}(t), a_{2}(t)]X\d t+\mathrm{diag}[b(t),b(t)]X\d W=: A(t)X\d t+F(t)X\d W,
\end{equation}
where the coefficients are defined as follows:
\begin{equation}\label{equ.18}
 a_{i}(t) =(-1)^{i}\times 	\left\{
	\begin{aligned}
	&-\alpha_{i},   &t\in[t_{2k-1}, t_{2k}),\\
	&\alpha_{i},  &t\in[t_{2k}, t_{2k+1}),\\
	\end{aligned}
\right.
\quad  b(t)\equiv\beta, \quad k\in \N, \quad i=1,2.
\end{equation}
It is obvious that $A(t)$ is a piecewise constant matrix and $F(t)\equiv F$ is a constant matrix,  where $\alpha_{1}$ and $\alpha_{2}$ will be given explicitly below.
According to the classical It\^o integral \cite{1976Stochastic}, we obtain the components of the  solution $X(t)=(X_{1}(t),X_{2}(t))$ with initial condition $X(1)=X_{0}=(1,1)$
$$X_{1}(t)=e^{\int_{1}^{t}a_{1}(\tau)\,\d\tau-\frac{1}{2}\beta^{2}t+\beta[W(t)-W(t_{0})]},$$
$$X_{2}(t)=e^{\int_{1}^{t}a_{2}(\tau)\,\d\tau-\frac{1}{2}\beta^{2}t+\beta[W(t)-W(t_{0})]},$$
which imply the value of Lyapunov exponents
\begin{equation}\label{equ.13}
	\hat{\lambda}_{i}=\varlimsup\limits_{t\rightarrow\infty}\frac{1}{t}\ln\Vert X_{i}(t)\Vert=\frac{\theta-1}{\theta+1}\alpha_{i}-\frac{1}{2}\beta^{2}, \quad i=1, 2.
\end{equation}
The limit $\lim\limits_{t\rightarrow\infty}\frac{W(t)}{t}=0$ is used in the above equality.
Then we can choose appropriate values of $\alpha_{1}$, $\alpha_{2} $ and $\beta$ to make the equalities $\hat{\lambda}_{i}=\lambda_{i}$, $i=1,2$ true.
\par
 \textbf{(2)} We consider the following equation  with deterministic initial value $X_{0}=(1,1)$:
 \begin{equation}\label{equ.61}
 \d Y(t)=[A(t)+Q(t)]Y(t)\d t,\quad Y(1)=X_{0},
 \end{equation}
 	where $Q(t)$ satisfies the condition \eqref{equ.3}.
We can take  $Q(t)$ such that the norm of $Y(t)$ and the angle $\gamma'_{2k-1}:=\angle \left\{Y(t_{2k-1}), Ox_{2}\right\}$  at time $t=t_{2k-1}$  satisfy the conditions
 $$	e^{\xi_{1} t_{2k-1}}\leq\Vert Y(t_{2k-1})\Vert \leq e^{\xi_{2}t_{2k-1}},$$
 $$\tan\gamma'_{2k-1}=e^{-\eta t_{2k-1}},\quad  \eta=-\theta\sigma+(\theta-1)(\alpha_{1}+\alpha_{2})>0,$$
 where $\xi_{1}=\frac{\theta\sigma}{\theta-1}-(\alpha_{2}+\frac{1}{2}\beta^{2})<0$  , $\xi_{2}=\frac{\theta\sigma}{\theta-1}-(\alpha_{1}+\frac{1}{2}\beta^{2})<0$  and $\angle \left\{Y(t_{2k-1}), Ox_{2}\right\}$ refers to the angle between $\overrightarrow{OY}(t_{2k-1})$ and the positive $x_{2}$-axis. For the perturbed SDE
\begin{equation}
	\d X=[A(t)+Q(t)]X\d t+F(t)X\d W,\quad X(1)=X_{0},
\end{equation}
 the solution $X(t)$ can be expressed by $X(t)=Y(t)e^{\beta W(t)}$, where $Y(t)$ satisfies \eqref{equ.61}.
So the norm of $X(t)$ and the angle $\gamma_{2k-1}:=\angle \left\{X(t_{2k-1}), Ox_{2}\right\}$ at time $t=t_{2k-1}$ satisfy the conditions
\begin{equation}\label{equ.62}
\begin{aligned}
&	e^{\xi_{1} t_{2k-1}}e^{\beta W(t_{2k-1})}\leq\Vert X(t_{2k-1})\Vert \leq e^{\xi_{2}t_{2k-1}}e^{\beta W(t_{2k-1})},\\
&\tan\gamma_{2k-1}=e^{-\eta t_{2k-1}}.
\end{aligned}
\end{equation}
 In the later proof, we will show that the estimate 
	$$	e^{\xi_{1}t_{2k}}e^{\beta W(t_{2k})}\leq\Vert X(t_{2k})\Vert \leq e^{\xi_{2}t_{2k}}e^{\beta W(t_{2k})}$$
is true,  and for any $\epsilon>0$ and almost all $\omega$, there exists a $k(\omega)$ such that for any $k>k(\omega)$ we have another estimate
$$t^{-1}\ln X_{1}(t,\omega)\leq t_{2k}^{-1}\ln X_{1}(t_{2k},\omega)+\epsilon\leq  t_{2k}^{-1}\ln \Vert X(t_{2k},\omega)\Vert+\epsilon,\quad t\in[t_{2k-1},t_{2k}].$$
These  imply the following inequality 
$$\xi_{1}\leq\varlimsup\limits_{t\rightarrow\infty}\frac{1}{t}\ln\Vert X(t)\Vert\leq\xi_{2},$$
so the Lyapunov exponent of $X(t)$ is negative.

\par  Then, we  construct the perturbation $Q(t)$ on the interval $[t_{2k-1},t_{2k})$. We define the  entries of the matrix $Q(t)=(q_{ij}(t))$ as follows:
$$q_{11}(t)=q_{12}(t)=q_{22}(t)=0,\quad t\in[t_{2k-1},t_{2k}],$$
\begin{equation}\label{equ.54}
	q_{21}(t)=d_{2k}e^{-\sigma t}\left\{
\begin{aligned}
	&0,   &t\in[t_{2k-1}, \tau_{1}),\\
	&e_{01}(t, \tau_{1},\tau_{2}),  &t\in[\tau_{1}, \tau_{2}),\\
	&1,   &t\in[\tau_{2}, \tau_{3}),\\
	&e_{10}(t, \tau_{3},\tau_{4}),  &t\in[\tau_{3}, \tau_{4}],\\
\end{aligned}
\right.
\end{equation}
	$$\epsilon(t)=e^{-t^{2}},\quad \tau_{1}=t_{2k}-1-2\epsilon(t_{2k}), \quad \tau_{2}=\tau_{1}+\epsilon(t_{2k}),  $$
$$\tau_{3}=\tau_{2}+1, \quad  \tau_{4}=\tau_{3}+\epsilon(t_{2k})=t_{2k},$$
where $d_{2k}$ will be given explicitly later. The function $e_{\alpha\beta}(\eta, \eta_{1},\eta_{2})$ is the $C^{\infty}$ Gelbaum--Olmsted function \cite{2019Construction}
\begin{align}\label{equ.60}
e_{\alpha\beta}(\eta, \eta_{1},\eta_{2})=\alpha+(\beta-\alpha)\exp\left\{-(\eta-\eta_{1})^{-2}\exp[-(\eta-\eta_{2})^{-2}]\right\},
\end{align}
\noindent
defined on the interval with endpoints $\eta=\eta_{1}$ and $\eta=\eta_{2}$ and takes the values
$$e_{\alpha\beta}(\eta_{1}, \eta_{1},\eta_{2})=\alpha, \quad e_{\alpha\beta}(\eta_{2}, \eta_{1},\eta_{2})=\beta ,$$
at these endpoints, which can polish the perturbation and equations.
\par The perturbed equation \eqref{equ.2} thereby can be expressed on the interval $[t_{2k-1},t_{2k})$ by two equations
$$\d X_{1}=\alpha_{1}X_{1}\d t+\beta X_{1}\d W,\quad \d X_{2}=[-\alpha_{2}X_{2}+q_{21}(t)X_{1}]\d t+\beta X_{2}\d W.$$
After integrating the equations, we obtain the following expressions for the components of the solution $X(t)=(X_{1}(t),X_{2}(t))$:
$$
X_{1}(t)=X_{1}(t_{2k-1})e^{(\alpha_{1}-\frac{1}{2}\beta^{2})(t-t_{2k-1})+\beta[W(t)-W(t_{2k-1})]}   ,\quad t\in [t_{2k-1},t_{2k}),$$
\begin{equation}\label{equ.42}
X_{2}(t)=X_{2}(t_{2k-1})e^{(-\alpha_{2}-\frac{1}{2}\beta^{2})(t-t_{2k-1})+\beta[W(t)-W(t_{2k-1})]}   ,\quad t\in [t_{2k-1},\tau_{1}],
\end{equation}
\begin{equation}\label{equ.43}
	\begin{aligned}	X_{2}(t)=&e^{(-\alpha_{2}-\frac{1}{2}\beta^{2})(t-\tau_{1})+\beta[W(t)-W(\tau_{1})]}  \\
\cdot \big[X_{2}(\tau_{1}&)+\int_{\tau_{1}}^{t} q_{21}(\tau)X_{1}(\tau)e^{(\alpha_{2}+\frac{1}{2}\beta^{2})(\tau-\tau_{1})-\beta[W(\tau)-W(\tau_{1})]}\d \tau \big] ,\quad t\in [\tau_{1},t_{2k}).
	\end{aligned}
\end{equation}
\par In one case, we suppose $d_{2k}=0$. Then $X_{2}(t)$ has expression \eqref{equ.42} for all $t\in[t_{2k-1},t_{2k})$. Therefore, in this case, $\gamma_{2k}=\angle\left\{X(t_{2k}),Ox_{1}\right\}$ satisfies the following formula
\begin{equation}\label{equ.59}
\begin{aligned}
\tan \gamma_{2k}&=\frac{X_{2}(t_{2k})}{X_{1}(t_{2k})}\\
&=\cot \gamma_{2k-1}e^{-(\alpha_{1}+\alpha_{2})(t_{2k}-t_{2k-1})}\cdot e^{\beta[W(t_{2k})-W(t_{2k-1})]-\beta[W(t_{2k})-W(t_{2k-1})]}\\
&=e^{-\sigma t_{2k}}.
\end{aligned}
\end{equation}
\par In the other case, for $d_{2k}=-d<0$, $X_{2}(t)$ has expression \eqref{equ.42} on $[t_{2k-1},\tau_{1}]$ and expression \eqref{equ.43} on $[\tau_{1},t_{2k}]$, and we will prove that it can vanish at time $t=t_{2k}$ for a $d_{2k}<0$ by the following fact
\begin{equation}\label{equ.7}
\begin{aligned}
&\tan \gamma_{2k}=\frac{X_{2}(t_{2k})}{X_{1}(t_{2k})}\\
&=\frac{X_{2}(\tau_{1})+\int_{\tau_{1}}^{t_{2k}} q_{21}(\tau)X_{1}(\tau)e^{(\alpha_{2}+\frac{1}{2}\beta^{2})(\tau-\tau_{1})-\beta[W(\tau)-W(\tau_{1})]}\d \tau }{X_{1}(\tau_{1})}\\
&\quad\cdot e^{(-\alpha_{2}-\frac{1}{2}\beta^{2})(t_{2k}-\tau_{1})+\beta[W(t_{2k})-W(\tau_{1})]}
\cdot e^{-(\alpha_{1}-\frac{1}{2}\beta^{2})(t_{2k}-\tau_{1})-\beta[W(t_{2k})-W(\tau_{1})]}.
\end{aligned}
\end{equation}
In order to simplify this expression, we only consider whether the first  component of \eqref{equ.7} vanishes at time $t=t_{2k}$,
\begin{equation}\label{equ.44}
\begin{aligned}
&\frac{X_{2}(\tau_{1})+\int_{\tau_{1}}^{t_{2k}} q_{21}(\tau)X_{1}(\tau)e^{(\alpha_{2}+\frac{1}{2}\beta^{2})(\tau-\tau_{1})-\beta[W(\tau)-W(\tau_{1})]}\d \tau }{X_{1}(\tau_{1})}\\
\leq& \frac{X_{2}(\tau_{1})}{X_{1}(\tau_{1})}+\int_{\tau_{1}}^{t_{2k}}q_{21}(\tau)e^{(\alpha_{1}+\alpha_{2})(\tau-\tau_{1})} \cdot e^{\beta[W(\tau)-W(\tau_{1})]-\beta[W(\tau)-W(\tau_{1})]}\d \tau \\
\leq&\cot \gamma_{2k-1}e^{-(\alpha_{1}+\alpha_{2})(\tau_{1}-t_{2k-1})}-d\int_{\tau_{2}}^{\tau_{3}}e^{-\sigma \tau+(\alpha_{1}+\alpha_{2})(\tau-\tau_{1})} \d \tau\\
\leq & e^{-\sigma t_{2k}+(\alpha_{1}+\alpha_{2})(t_{2k}-\tau_{1})}-cde^{-\sigma\tau_{1}}\\
\leq & e^{-\sigma \tau_{1}}(e^{(\alpha_{1}+\alpha_{2}-\sigma)(t_{2k}-\tau_{1})}-cd)\leq 0,
\end{aligned}
\end{equation}
where $c=\frac{e^{(\alpha_{2}+\alpha_{1}-\sigma)(\tau_{3}-\tau_{1})}-e^{(\alpha_{2}+\alpha_{1}-\sigma)(\tau_{2}-\tau_{1})}}{\alpha_{2}+\alpha_{1}-\sigma}$.
These inequalities hold by virtue of the requirement $\sigma<\alpha_{1}+\alpha_{2}$.  Then  estimate \eqref{equ.44} is equivalent to the inequality $\tan \gamma_{2k}\leq 0$.  From the inequality, we see that the angle  $\gamma_{2k}$ vanishes when $d=\frac{1}{c}e^{(\alpha_{1}+\alpha_{2}-\sigma)(t_{2k}-\tau_{1})}$. And $\tan \gamma_{2k}$ is independent of $\omega$.  Therefore, by the continuous dependence of $\tan \gamma_{2k}$ on $d_{2k}$ and estimates \eqref{equ.59}, \eqref{equ.7} and \eqref{equ.44}, there exists a deterministic $d_{2k}\leq d$ such that the identity $\tan\gamma_{2k}=e^{-\eta t_{2k}}$ holds.  It is immediate to check that a deterministic appropriate $\hat{d}_{2k}$ can be chose to let this angle vanish at $t_{2k}$ which will be used in the proof of Theorem \ref{the.2}.
\par
Finally, we will show that the norm of the solution $X(t)$ on the interval $[t_{2k-1},t_{2k})$ meets the requirements. Since $\xi_{2}+\alpha_{2}+\frac{1}{2}\beta^{2}>0$, we have the estimates for the second component
\begin{equation}\label{equ.9}
	\begin{aligned}
		0&<\vert X_{2}(t)\vert\leq\vert X_{2}(t_{2k-1})\vert e^{(-\alpha_{2}-\frac{1}{2}\beta^{2})(t-t_{2k-1})+\beta[W(t)-W(t_{2k-1})]}\\
		&\leq e^{\xi_{2}t-(\xi_{2}+\alpha_{2}+\frac{1}{2}\beta^{2})(t-t_{2k-1})}e^{\beta W(t)}\\
		&\leq e^{\xi_{2}t}e^{\beta W(t)}.
	\end{aligned}
\end{equation}
\par
At the same time, for $X_{1}(t_{2k})$, we have the following estimate
\begin{equation}\label{equ.11}
	\begin{aligned}
	&t_{2k}^{-1}\ln X_{1}(t_{2k})\\
		\leq&t_{2k}^{-1}\ln[X_{2}(t_{2k-1})\tan\gamma_{2k-1}e^{(\alpha_{1}-\frac{1}{2}\beta^{2})(t_{2k}-t_{2k-1})}e^{\beta[W(t_{2k})-W(t_{2k-1})]}]\\
		\leq &t_{2k}^{-1}\ln[e^{\xi_{2}t_{2k-1}}e^{\beta W(t_{2k-1})}e^{-\eta t_{2k-1}}]\\
		&+t_{2k}^{-1}\ln[e^{(\alpha_{1}-\frac{1}{2}\beta^{2})(t_{2k}-t_{2k-1})}e^{\beta[W(t_{2k})-W(t_{2k-1})]}]\\
		\leq&t_{2k}^{-1}\ln e^{[\xi_{2}-\eta+(\theta-1)(\alpha_{1}-\frac{1}{2}\beta^{2})]t_{2k-1}+\beta W(t_{2k})}\\
		\leq&\xi_{2}+\frac{\beta W(t_{2k})}{t_{2k}}.
	\end{aligned}
\end{equation}
In the last inequality, we need $\xi_{2}-\eta+(\theta-1)(\alpha_{1}-\frac{1}{2}\beta^{2})\leq \theta\xi_{2}$, which holds by our choice $\xi_{2}=\frac{\theta\sigma}{\theta-1}-(\alpha_{1}+\frac{1}{2}\beta^{2})$.
\par Now we make an estimate of $X_{1}(t)$, $t\in[t_{2k-1},t_{2k})$. Note that this step is not a part of the induction, so it has no effect on the induction. For the first component $X_{1}(t)$ of the solution $X(t)$ on the interval $[t_{2k-1},t_{2k})$,   we have the following  expression
\begin{equation}\label{equ.10}
	\begin{aligned}
	t^{-1}\ln X_{1}(t)=&t^{-1}[(\alpha_{1}-\frac{1}{2}\beta^{2})(t-t_{2k-1})]+t^{-1}\beta[W(t)-W(t_{2k-1})]\\
		&+t^{-1}\ln X_{1}(t_{2k-1}),
	\end{aligned}	
\end{equation}
where $k$ is big enough and $0<X_{1}(t_{2k-1})<1$. Note that the first and third terms on the right-hand side of \eqref{equ.10} are monotonously increasing with respect to $t$. Now we consider the solutions from the pathwise viewpoint. According to the fact $\lim\limits_{t\rightarrow\infty}\frac{W(t)}{t}=0$ $a.s.$,  for any $\epsilon>0$, there exists a $T(\omega)$ such that  $\vert W(t,\omega)\vert<\frac{\epsilon}{4\beta} t$, as $t>T(\omega)$. For almost all $\omega$, there exists a $k(\omega)$ such that  $t_{2k(\omega)-1}\geq T(\omega)$.  So for any $t_{1},t_{2}\in [t_{2k(\omega)-1},t_{2k(\omega)})$ and $t_{1}>t_{2}$, the second term has the following estimate:
$$\vert t_{1}^{-1}\beta[W(t_{1},\omega)-W(t_{2},\omega)]\vert<\frac{1}{2}\epsilon. $$ 
We denote the maximum point of \eqref{equ.10} by $t_{\max}$. We get a relationship
$$	t_{\max}^{-1}\ln X_{1}(t_{\max},\omega)-	t_{2k}^{-1}\ln X_{1}(t_{2k},\omega)<\epsilon.$$
Therefore, we have the estimate
$$t^{-1}\ln X_{1}(t,\omega)\leq	t_{\max}^{-1}\ln X_{1}(t_{\max},\omega)<t_{2k}^{-1}\ln X_{1}(t_{2k},\omega)+\epsilon,\quad t\in[t_{2k(\omega)-1},t_{2k(\omega)}).$$
\quad Next, we continue to prove the induction.  In the same way as \eqref{equ.11}, we get the following lower bound
\begin{equation}\label{equ.12}
\begin{aligned}
	t_{2k}^{-1}\ln X_{1}(t_{2k}) &\geq t_{2k}^{-1}\ln \big[e^{[\xi_{1}-\eta+(\theta-1)(\alpha_{1}-\frac{1}{2}\beta^{2})]t_{2k-1}+\beta W(t_{2k})}\big]\\
	&\geq\xi_{1}+\frac{\beta W(t_{2k})}{t_{2k}}.
\end{aligned}
\end{equation}
\par In conclusion, by the inequalities $0<X_{2}(t_{2k})<X_{1}(t_{2k})$ and \eqref{equ.9}--\eqref{equ.12},  we obtain the desired estimates
$$	e^{\xi_{1}t_{2k}}e^{\beta W(t_{2k})}\leq\Vert X(t_{2k})\Vert \leq e^{\xi_{2}t_{2k}}e^{\beta W(t_{2k})},$$
and we have   another estimate,  for almost all $\omega$ and any $k>k(\omega)$,
$$t^{-1}\ln X_{1}(t,\omega)\leq t_{2k}^{-1}\ln X_{1}(t_{2k},\omega)+\epsilon\leq  t_{2k}^{-1}\ln \Vert X(t_{2k},\omega)\Vert+\epsilon,\quad t\in[t_{2k-1},t_{2k}).$$
 \quad On the next interval $[t_{2k},t_{2k+1})$, by the similar method above, we can get the estimates of the norm of $X(t)$ and angle $\gamma_{2k+1}=\angle \left\{X(t_{2k+1}),Ox_{2}\right\}$ at $t_{2k+1}$ as follows:
$$	e^{\xi_{1} t_{2k+1}}e^{\beta W(t_{2k+1})}\leq\Vert X(t_{2k+1})\Vert \leq e^{\xi_{2}t_{2k+1}}e^{\beta W(t_{2k+1})},$$
$$\tan\gamma_{2k+1}=e^{-\eta t_{2k+1}}, \quad \eta=-\theta\sigma+(\theta-1)(\alpha_{1}+\alpha_{2})>0.$$
During the proof of the above esimate, the following inequality is required:
\begin{equation}
	\begin{aligned}
t_{2k+1}^{-1}\ln X_{2}(t_{2k+1})\leq&t_{2k+1}^{-1}\ln e^{[\xi_{2}-\eta+(\theta-1)(\alpha_{2}-\frac{1}{2}\beta^{2})]t_{2k}+\beta W(t_{2k+1})}\\
		\leq&\xi_{2}+\frac{\beta W(t_{2k+1})}{t_{2k+1}}, \quad t\in[t_{2k},t_{2k+1}),
	\end{aligned}
\end{equation}
which holds by our choice $\xi_{2}=\frac{\theta\sigma}{\theta-1}-(\alpha_{1}+\frac{1}{2}\beta^{2})$. And for any $\epsilon>0$ and almost all $\omega$, there exists a $k(\omega)$ such that for any $k>k(\omega)$, we have the estimate
$$t^{-1}\ln X_{2}(t,\omega)\leq t_{2k+1}^{-1}\ln X_{2}(t_{2k+1},\omega)+\epsilon\leq t_{2k+1}^{-1}\ln \Vert X(t_{2k+1},\omega)\Vert+\epsilon,\quad t\in[t_{2k},t_{2k+1}).$$
The constructed perturbation on $[t_{2k},t_{2k+1})$ can be expressed by the matrix:
$$
	Q(t)=\left(
\begin{array}{cc}
	0 & q_{12}(t) \\
	0 & 0
\end{array}\right),\quad t\in[t_{2k},t_{2k+1}),$$
where $q_{12}(t)$ is defined similarly as  \eqref{equ.54}.
\par
By induction we extend these constructions of perturbed equation \eqref{equ.2} to get a $C^{\infty}$ perturbation and a desired solution $X_{0}$ with exponent $\lambda(X)=\lambda_{0}<0$ on the interval $[t_{1},\infty)$ where $\xi_{1}\leq\lambda_{0}\leq \xi_{2}=\lambda(\theta,\lambda_{1},\sigma)$. But $A$ still has countably many discontinuity points  $t=t_{k}$. By the similar function \eqref{equ.60}, we polish function $A$. Then we will get a $C^{\infty}$ coefficient $B$. The $C^{\infty}$ equation
$$	\d X=B(t)X\d t+F(t)X\d W$$
has same Lyapunov exponents with original equation \eqref{equ.1} according to the definition of Lyapunov exponents. We also can use the similar method as above to prove the Lyapunov exponents of the perturbed equation
$$	\d X=[B(t)+Q(t)]X\d t+F(t)X\d W $$
 are same with respect to  equation \eqref{equ.2}.  So the desired equation is constructed.
\par Finally we will prove the uniqueness of the solution $X_{0}(t)$ with negative exponent of  perturbed linear equation \eqref{equ.2}.  Therefore, assuming the opposite --- there exists a second solution $X_{1}(t)$ with negative exponent that is linearly independent of $X_{0}(t)$ --- we would have the following contradiction according to the Lyapunov inequality \cite{1998Stochastic} and Liouville's theorem \cite{Vrko1978Liouville}:
\begin{equation}\label{equ.58}
\begin{aligned}
0>\lambda[X_{0}]+\lambda[X_{1}]\geq&\varlimsup\limits_{t\rightarrow\infty}\frac{1}{t}\ln\vert \det\Phi(t)\vert\\
\geq&\varlimsup\limits_{t\rightarrow\infty}\frac{1}{t}\ln e^{\int_{1}^{t}tr(A+Q)\d\tau-\frac{1}{2}\int_{1}^{t}tr(F^{2})\d\tau+\int_{1}^{t}trF\d W}\\
=&(\alpha_{2}-\alpha_{1})\frac{\theta-1}{\theta+1}-\beta^{2}\geq0.
\end{aligned}
\end{equation}
Let $\beta$ small enough satisfying $(\alpha_{2}-\alpha_{1})\frac{\theta-1}{\theta+1}-\beta^{2}\geq0$. So $X_{0}$ is the unique solution with negative Lyapunov exponent by the contradiction.  Note that the above inequality proves the third part of this theorem. The proof is complete.
\end{proof}
\par The inequality \eqref{equ.58} implies the following assertion, which will be useful to prove Theorem \ref{the.2}.\\
\textbf{Assertion.} For the solution $X_{1}(t)$ and for any sequence $\left\{k(l)\right\}$ of odd numbers with the property $\frac{k(l)}{k(l+1)}\rightarrow0$ as $l\rightarrow\infty$, we have the inequality
$$
\varlimsup\limits_{l\rightarrow\infty}\frac{1}{t_{k(l+1)}}\ln \frac{\Vert X_{1}(t_{k(l+1)})\Vert}{\Vert X_{1}(t_{k(l)})\Vert}\geq -\lambda_{0}>0.$$

\subsection{$n$-dimensional situation}
\par
We consider $n$-dimensional situation. The next questions are whether  Theorem \ref{the.1} still holds when the dimension of equation \eqref{equ.2} is $n$  and  the possible number of linearly independent solutions with negative Lyapunov exponents. By the similar method, we have the following result.
\par
\begin{theorem}\label{the.2}
	\begin{itemize}
      \item [(1)] For any constants $\lambda_{n}\geq\ldots\geq\lambda_{2}\geq\lambda_{1}>0$, $n\geq3,$ $\theta>1$,
      there exists an $n$-dimensional linear equation \eqref{equ.1} with bounded $C^{\infty}$ coefficients and Lyapunov exponents $\lambda_{i}$, $i=1,2,\ldots,n$.
      \par \item [(2)] For this equation and any real number $  0<\sigma<\frac{(\theta-1)\lambda_{1}}{\theta}$, there exists a $C^{\infty}$ exponentially decaying perturbation $Q$ such that  the $n$-dimensional perturbed equation \eqref{equ.2} has exactly $n-1$ linearly independent solutions
      $$X_{1},X_{2},\ldots,X_{n-1}$$
      with negative exponents $\Lambda_{i},$ $i=1,\ldots,n-1$ which is defined similarly as in Theorem \ref{the.1};
     the Lyapunov exponent of the other  solution $X_{n}$ which is linearly independent of  $X_{i}, i=1,\ldots,n-1$ is positive and is greater than $\vert\sum_{i=1}^{n-1}\Lambda_{i}\vert$.
	\end{itemize}
\end{theorem}
\par \begin{proof}
\textbf{(1)} First, we construct $n$-dimensional equation \eqref{equ.1} according to Theorem \ref{the.1}. For the sequence $\left\{k(l)\right\}$ of odd numbers with the property $\frac{k(l)}{k(l+1)}\rightarrow0$, as $l\rightarrow\infty$, we define numbers $T_{l}=t_{k(l)}$, $l\in \N$ based on the times $t_{k}=\theta^{k}$, $\theta>1$, $k\in \N$.  By the given exponents $\lambda_{i}>0$,  we use the similar method in  Theorem \ref{the.1} to get the
numbers $\alpha_{n}\geq\ldots\geq\alpha_{2}\geq\alpha_{1}, \beta$. With these numbers, we define the piecewise constant coefficients $a_{1}(t),a_{2}(t),\ldots,a_{n}(t),b(t)$ of the linear $n$-dimensional diagonal equation
\begin{equation}\label{equ.15}
\d X=\mathrm{diag}[a_{1}(t),\ldots,a_{n}(t)]X\d t+\mathrm{diag}[b(t),\ldots,b(t)]X\d W,\quad t\geq T_{1}
\end{equation}
as follows. The coefficients $a_{1}(t)$ and $b(t)$ will be defined  on the  interval $[T_{1},+\infty]$ by relations (\ref{equ.18}). The remaining coefficients $a_{i}(t)$ will be defined by the relations
\begin{center}
$	a_{i}(t)=-\alpha_{i}\times  \text{sgn } a_{1}(t),\quad t\in[T_{i-1+k(n-1)},T_{i+k(n-1)}],\quad k\in \N_{+},\quad i=2,\ldots,n,$\end{center}
and on other intervals in each period, by the relations
$$a_{i}=-\alpha_{i},\quad t\in[T_{i+k(n-1)},T_{i-1+(k+1)(n-1)}],\quad k\in \N_{+},\quad i=2,\ldots,n,$$
which ensure the inequalities $\Lambda_{i}+\alpha_{i}+\frac{1}{2}\beta^{2}>0$.
\par \textbf{(2)} On the interval $[T_{1},T_{2}]$,  using the constructions in the proof of Theorem \ref{the.1}, in the coordinate plane $x_{1}Ox_{2}$,  we obtain the first desired solution $X_{1}(t)=(x(t),0,\ldots,0)$ with $\Lambda_{1}$--exponentially decaying property where $x(t)$ is a 2-dimensional random vector. Next, we use an appropriate perturbation (see the proof of Theorem \ref{the.1}) on the interval $[T_{2}-1,T_{2}]$ such that the solution $X_{1}(t)$ achieves at the coordinate axis $Ox_{2}$ at $t=T_{2}$. Thereby, we obtain the expression $X_{1}(T_{2})=\Vert X_{1}(T_{2})\Vert \vec{e}_{2}$, where $\vec{e}_{i}$ is the $i$th coordinate vector of the $n$-dimensional space. Then, the solution will be determined by the relation
$$X_{1}(t)=\Vert X_{1}(T_{2})\Vert \vec{e}_2\cdot e^{-(\alpha_{2}+\frac{1}{2}\beta^{2})(t-T_{2})}e^{\beta[W(t)-W(T_{2})]},\quad  t\in[T_{2},T_{n}],$$
which will not be influenced by perturbations on $[T_{2},T_{n}]$.
\par The second solution $X_{2}(t)$ of equation (\ref{equ.15}) has the expression
$$X_{2}(t)=\vec{e}_{3}\cdot e^{-(\alpha_{2}+\frac{1}{2}\beta^{2})(t-T_{1})}e^{\beta[W(t)-W(T_{1})]},\quad t\in[T_1,T_2],$$
which is not affected by\label{key} perturbations on the interval $[T_{1},T_{2}]$. There are three procedures to construct the  solution on the interval $[T_{2},T_{3}]$. Firstly, on the interval $[T_{2},T_{2}+1]$, we form  the desired norm of $X_{2}(T_{2}+1)$ and angle  $\angle\left\{X_{2}(T_{2}+1),Ox_{3}\right\}$ by an appropriate perturbation  acting in the plane $x_{1}Ox_{3}$, which satisfy the similar condition \eqref{equ.61} of the induction. Secondly, in accordance with the constructions and arguments in the proof of Theorem \ref{the.1}, on the interval $[T_{2}+1,T_{3}-1]$, we obtain desired solution with $\Lambda_{2}$--exponentially decaying property. Finally, we add a perturbation on interval $[T_{3}-1,T_{3}]$ such that  the solution $X_{2}(t)$ achieves at the coordinate axis $Ox_{3}$ at $t=T_{3}$.
Notice that the solution will not be affected by other perturbations on $[T_{3},T_{n}]$. The other solutions has similar construction. In order to make it easier to understand, we give  perturbation $Q(t)$ on the interval $[T_{1},T_{3}]$ in 3-dimensional situation:
$$
Q(t)=\left(\begin{matrix}
	0 & q_{12}(t) & q_{13}(t) \\
	q_{21}(t) & 0 & 0 \\
	q_{31}(t) & 0 & 0
\end{matrix}\right),
$$
$$
	q_{13}(t)=q_{31}(t)=0,\quad t\in [T_{1},T_{2}], \quad	q_{12}(t)=q_{21}(t)=0,\quad t\in [T_{2},T_{3}].
$$
 $q_{12}$, $q_{21}$ on the interval $[T_{1},T_{2}]$ and  $q_{13}$, $q_{31}$  on the interval $[T_{2},T_{3}]$ have similar definition as the component of \eqref{equ.54} in Theorem \ref{the.1}. 
\par By induction, the equation (\ref{equ.2})  already has $n-1$
linearly independent solutions  with negative exponents. Now let us prove that it cannot have another solution $X_{n}(t)$  with  negative
exponent which is linearly independent with  solutions $X_{1},X_{2},\ldots,X_{n-1}$. Suppose the contrary -- equation \eqref{equ.2} has such a solution $X_{n}(t)$. Then there are $n$ linearly independent solutions with negative Lyapunov exponents. So according to the \textbf{Assertion}, we have

\begin{align*}
	0 & >\lambda[X_{1}]+\ldots+\lambda[X_{n}]\\
	&\geq\varlimsup\limits_{l\rightarrow\infty}\frac{1}{t_{k(l)}}\ln\Vert X_{1}(t_{k(l)})\Vert+\ldots+\varlimsup\limits_{l\rightarrow\infty}\frac{1}{t_{k(l+1)}}\ln\Vert X_{p(l+1)}(t_{k(l+1)})\Vert\\
	&\quad+\ldots+\varlimsup\limits_{l\rightarrow\infty}\frac{1}{t_{k(l)}}\ln\Vert X_{n}(t_{k(l)})\Vert\\
	&\geq\varlimsup\limits_{l\rightarrow\infty}\frac{1}{t_{k(l+1)}}\ln\big[\Vert  X_{p(l+1)}(t_{k(l+1)})\Vert\times\prod_{i\neq p(l+1)}\Vert X_{i}(t_{k(l)})\Vert\big]\\
	&\geq \varlimsup\limits_{l\rightarrow\infty}\frac{1}{t_{k(l+1)}}\ln\big[ \frac{\Vert X_{p(l+1)}(t_{k(l+1)})\Vert}{\Vert X_{p(l+1)}(t_{k(l)})\Vert}\times \vert \det \Phi(t_{k(l)})\vert\big]\\
	&\geq-\Lambda_{1}>0,
\end{align*}
where $p(l+1)\in\{1,...,n-1\}$ is the subscript of the exponentially increasing solution on the interval $[T_{l},T_{l+1}]$ in  Theorem \ref{the.1}  and $\Lambda_{1}$ is the smallest negative Lyapunov exponent.  The contradiction implies that the perturbed equation only has $n-1$ linearly independent solutions with negative Lyapunov exponents. And the positive Lyapunov exponent is greater than $\vert\sum_{i=1}^{n-1}\Lambda_{i}\vert$.
\par  At the end of proof, the similar method can be used to prove that the Lyapunov exponents are invariant after polishing the coefficient matrices $A(t)$ in the $n$-dimensional situation. The proof is complete.
\end{proof}
\begin{remark}\rm
We remark that there exist an equation \eqref{equ.1} which has $n-1$ linearly independent solutions with negative Lyapunov exponents and another solution with positive Lyapunov exponent.    The equation \eqref{equ.2} will have all positive Lyapunov exponents under an appropriate preturbation $Q(t)$.
\end{remark}

\subsection{The situation of all negative Lyapunov exponents}
\par For equations \eqref{equ.1} with all negative Lyapunov exponents, we  have the following theorem. Its linear ordinary differential equations version was proved by Izobov \cite{2019Construction}.
\par
\begin{theorem}\label{the.3}
	\begin{itemize}
      \par \item [(1)] For any constants $\lambda_{2}\leq\lambda_{1}<0$, $\theta>1$, there exists a $2$-dimensional equation \eqref{equ.1} with bounded $C^{\infty}$ coefficients and Lyapunov exponents $\lambda_{1}, \lambda_{2}$.
      \par \item [(2)] For the equation above, there exists $h:=h(\lambda_1,\lambda_2,\theta)$ such that for any $\sigma\in (0,h)$, there exists a $C^{\infty}$ exponentially decaying perturbation $Q(t)$ such that the perturbed equation \eqref{equ.2} has a unique solution with positive Lyapunov exponent.
	\end{itemize}
\end{theorem}
\begin{proof}
	\textbf{(1)}
	For any $\theta>1$, there exist $p_{2}>p_{1}>0$ which will be given explicitly later. We set $t_{0}=p_{1}+p_{2}$ and denote the length of interval $[t_{2k-2},t_{2k-1}]$ by $p_{2}(\theta^{k}-\theta^{k-1})$ and that of interval $[t_{2k-1},t_{2k}]$ by $p_{1}(\theta^{k}-\theta^{k-1})$. Then we get $t_{2k}=(p_{1}+p_{2})\theta^{k}$. Finally, let us consider the following 2-dimensional equation with deterministic  initial value $X_{0}=(1,1)$
	\begin{equation}\label{equ.20}
    \d X=\mathrm{diag}[a_{1}(t),a_{2}(t)]X\d t+\mathrm{diag}[b(t),b(t)]X\d W, \quad t\geq t_{0},
	\end{equation}
where  the coefficients are defined as follows:
	$$
	a_{1}(t)=-\alpha_{1}, \quad a_{2}(t)= \left\{
	\begin{aligned}
		&-\alpha_{2}, \quad&t\in[t_{2k-2},t_{2k-1}),\\
		&\alpha_{2}, \quad&t\in[t_{2k-1},t_{2k}),
	\end{aligned}
	\right.
	\quad b(t)=\beta.$$
It is immediate to find that the Lyapunov exponents of equation \eqref{equ.20} are $-\alpha_{1}-\frac{1}{2}\beta^{2}$ and $\frac{p_{1}-p_{2}}{p_{1}+p_{2}}\alpha_{2}-\frac{1}{2}\beta^{2}$. For the given $\lambda_{1},\lambda_{2}$, we choose appropriate positive constants $p_1,p_2,\alpha_{1},\alpha_{2},\beta$ such that  $$\lambda_{1}=-\alpha_{1}-\frac{1}{2}\beta^{2},\quad \lambda_{2}=\frac{p_{1}-p_{2}}{p_{1}+p_{2}}\alpha_{2}-\frac{1}{2}\beta^{2},$$ 
\begin{equation}\label{equ.55}
\frac{p_2-p_1}{p_1}<\frac{\alpha_1+\alpha_{2}}{\theta\alpha_2},\quad\dfrac{(\theta-1)p_{1}\alpha_{1}}{\theta(p_{2}-p_{1})}>\alpha_{1}+\dfrac{1}{2}\beta^{2},
\end{equation}
by taking $p_{2}-p_{1}$ small enough and $\alpha_{2}\gg p_{1}$ with $p_{1}$ big enough.
Due to $\lambda_{2}\leq\lambda_{1}$, we get the inequality
\begin{equation}\label{equ.56}
\frac{\alpha_{2}}{p_{1}+p_{2}}>\frac{\alpha_{1}}{p_{2}-p_{1}},
\end{equation}
which is useful later.

\par \textbf{(2)} According to the exponentially decaying property, we  can choose an appropriate perturbation on the interval $[t_{0},t_{2k-2}]$ such that the angle $\gamma_{2k-2}:=\angle \left\{X(t_{2k-2}), Ox_{1}\right\}$  satisfies the condition
$$
\tan\gamma_{2k-2}=e^{\eta t_{2k-2}}, \quad \eta=\theta\sigma_{1}-(\alpha_{1}+\frac{p_{1}-p_{2}}{p_{1}+p_{2}}\alpha_{2})(\theta-1)>0,
$$
where  $\sigma_{1}>0$.  It's obvious that $\eta>\sigma_{1}$. The lower bound of $\eta$ with respect to $\sigma_{1}$ is $-(\alpha_{1}+\frac{p_{1}-p_{2}}{p_{1}+p_{2}}\alpha_{2})(\theta-1)$
satisfying
$$-(\alpha_{1}+\frac{p_{1}-p_{2}}{p_{1}+p_{2}}\alpha_{2})(\theta-1)<\frac{p_{1}(\theta-1)(\alpha_{1}+\alpha_{2})}{\theta(p_{1}+p_{2})},$$
due to the first inequality of \eqref{equ.55}. This allows us to take an  appropriate $\sigma_{1}$ such that $\eta\in(\alpha_{1}+\frac{1}{2}\beta^{2},\frac{p_{1}(\theta-1)(\alpha_{1}+\alpha_{2})}{\theta(p_{1}+p_{2})})$ where
$$ \alpha_{1}+\frac{1}{2}\beta^{2}<\frac{p_{1}(\theta-1)(\alpha_{1}+\alpha_{2})}{\theta(p_{1}+p_{2})} $$
are guaranteed by the second inequality of \eqref{equ.55} and \eqref{equ.56}. In conclusion, we get an appropriate value of $\eta$.

\par  We add a perturbation $Q(t)$ which is similar with the one in Theorem \ref{the.1} to equation \eqref{equ.20} on the interval $[t_{2k-2},t_{2k-1})$. The perturbation is defined as follows:
	$$q_{11}(t)=q_{12}(t)=q_{22}(t)=0,\quad t\in[t_{2k-2},t_{2k-1}),$$
	$$	q_{21}(t)=d_{2k}e^{-\sigma t}\left\{
	\begin{aligned}
		&0,   &t\in[t_{2k-2}, \tau_{1}),\\
		&e_{01}(t, \tau_{1},\tau_{2}),  &t\in[\tau_{1}, \tau_{2}),\\
		&1,   &t\in[\tau_{2}, \tau_{3}),\\
		&e_{10}(t, \tau_{3},\tau_{4}),  &t\in[\tau_{3}, \tau_{4}),\\
	\end{aligned}
	\right.$$
	$$\epsilon(t)=e^{-t^{2}},\quad \tau_{1}=t_{2k-1}-1-2\epsilon(t_{2k}), \quad \tau_{2}=\tau_{1}+\epsilon(t_{2k}),  $$
	$$\tau_{3}=\tau_{2}+1, \quad  \tau_{4}=\tau_{3}+\epsilon(t_{2k})=t_{2k-1}.$$
On the next interval $[t_{2k-1},t_{2k})$, 	$Q(t)=(q_{ij}(t))=0$.

By the same way, the components of  solution $X(t)=(X_{1}(t),X_{2}(t))$ of the perturbed equation \eqref{equ.2} can be expressed as follows:
\begin{align*}
		&X_{1}(t)=X_{1}(t_{2k-2})e^{(-\alpha_{1}-\frac{1}{2}\beta^{2})(t-t_{2k-2})+\beta[W(t)-W(t_{2k-2})]}, &t\in [t_{2k-2},t_{2k}),\\
		&X_{2}(t)=X_{2}(t_{2k-2})e^{(-\alpha_{2}-\frac{1}{2}\beta^{2})(t-t_{2k-2})+\beta[W(t)-W(t_{2k-2})]}, &t\in
			[t_{2k-2},\tau_{1}],\\ &X_{2}(t)=e^{(-\alpha_{2}-\frac{1}{2}\beta^{2})(t-\tau_{1})+\beta[W(t)-W(\tau_{1})]}\cdot\Big[X_{2}(\tau_{1})\\
			&\qquad +\int_{\tau_{1}}^{t} q_{21}(\tau) X_{1}(\tau)e^{(\alpha_{2}+\frac{1}{2}\beta^{2})(\tau-\tau_{1})-\beta[W(\tau)-W(\tau_{1})]}\d\tau \Big],  &t\in [\tau_{1},t_{2k-1}],\\
			&X_{2}(t)=X_{2}(t_{2k-1})e^{(\alpha_{2}-\frac{1}{2}\beta^{2})(t -t_{2k-1})+\beta[W(t)-W(t_{2k-1})]},  &t\in [t_{2k-1},t_{2k}).
\end{align*}
If $d_{2k}=0$, the angle $\gamma_{2k}$ has the following formula by equation \eqref{equ.20}:
\begin{align*}
		\tan\gamma_{2k}&=\frac{X_{2}(t_{2k})}{X_{1}(t_{2k})}
		=e^{\eta t_{2k-2}}e^{\alpha_{1}(t_{2k}-t_{2k-2})}e^{\alpha_{2}\frac{p_{1}-p_{2}}{p_{1}+p_{2}}(t_{2k}-t_{2k-2})}
		=e^{\sigma_{1} t_{2k}}.
	\end{align*}
If $d_{2k}=d>0$, it is not difficult to get the following estimate from Theorem \ref{the.1}:
	\begin{equation}\label{equ.23}
		\begin{aligned}
			\tan \gamma_{2k}&=\frac{X_{2}(t_{2k})}{X_{1}(t_{2k})}\\
			&=\frac{X_{2}(\tau_{1})+\int_{\tau_{1}}^{t_{2k-1}} q_{21}(\tau)X_{1}(\tau)e^{(\alpha_{2}+\frac{1}{2}\beta^{2})
					(\tau-\tau_{1})-\beta[W(\tau)-W(\tau_{1})]}\d \tau }{X_{1}(\tau_{1})} \\
			&\quad\cdot Ce^{(\alpha_{1}+\alpha_{2})(t_{2k}-t_{2k-1})}\\
			&\geq Ce^{(\alpha_{1}+\alpha_{2})(t_{2k}-t_{2k-1})}\big[\frac{X_{2}(\tau_{1})}{X_{1}(\tau_{1})}+\int_{\tau_{2}}^{\tau_{3}}de^{-\sigma \tau}e^{(\alpha_{2}-\alpha_{1})(\tau-\tau_{1})}\d \tau\big]\\
			&\geq Ce^{(\alpha_{1}+\alpha_{2})(t_{2k}-t_{2k-1})} \frac{e^{(\alpha_2-\alpha_1-\sigma)(\tau_{3}-\tau_{1})}-e^{(\alpha_2-\alpha_1-\sigma)(\tau_{2}-\tau_{1})}}{\alpha_2-\alpha_1-\sigma}de^{-\sigma\tau_{1}}\\
			&\geq e^{\eta t_{2k}},
	\end{aligned}
\end{equation}
	where $\sigma\in(0,\frac{p_{1}(\theta-1)(\alpha_{1}+\alpha_{2})}{\theta(p_{1}+p_{2})})-\eta)$ and $C>0$ are constants. We denote $\frac{p_{1}(\theta-1)(\alpha_{1}+\alpha_{2})}{\theta(p_{1}+p_{2})})-\eta$ by $h(\lambda_1,\lambda_2,\theta)$. By the continuous dependence of $\tan \gamma_{2k}$ on $d_{2k}$, we  choose a constant $d_{2k}\leq d$  such that the angle satisfies $\tan \gamma_{2k}=e^{\eta t_{2k}}$. By the similar method in Theorem \ref{the.1}, we find that for any $\epsilon>0$ and almost all $\omega$, there exists a $k(\omega)$ such that for any $k>k(\omega)$, we have the estimate
		$$t^{-1}\ln X(t)\leq t_{2k}^{-1}\ln \Vert X(t_{2k})\Vert+\epsilon,\quad t\in[t_{2k-2},t_{2k}).$$
	And we get the following formula
$$X_{2}(t_{2k})=\tan \gamma_{2k}\cdot X_{1}(t_{2k})=e^{(\eta-\alpha_{1}-\frac{1}{2}\beta^{2})t_{2k}+\beta W(t_{2k})},$$
	which implies that the Lyapunov exponent of solution $X(t)$ is $\eta-\alpha_{1}-\frac{1}{2}\beta^{2}>0$.
	Polishing the coefficients and the proof of uniqueness of the solution can be found in the proof of Theorem \ref{the.1}. 
\end{proof}
\subsection{The second form of perturbation}
In this subsection, we discuss the second situation that the perturbation appears at the inhomogeneous term. The perturbed equation can be expressed by  equation (\ref{equ.50}), which will have the similar conclusion about Lyapunov exponents. It is shown by same techniques. And $\textbf{Q}$, $\textbf{q}$ mean that they are random.
\begin{theorem}\label{the.7}
	\begin{itemize}
      \item [(1)] For any constants $\lambda_{2}\geq\lambda_{1}>0$, $\theta>1$, there exists a $2$-dimensional equation \eqref{equ.1} with bounded $C^{\infty}$ coefficients and Lyapunov exponents $\lambda_{1},\lambda_{2}$.
      \item [(2)] For this equation and any constant  $\sigma\in(0,\frac{(\theta-1)\lambda_{1}}{\theta})$, there exists a random perturbation $\textbf{Q}$ satisfying that for almost all $\omega$, $\textbf{Q}(t,\omega)$ is exponentially decaying, such that the perturbed equation \eqref{equ.50} has a  solution with negative Lyapunov exponent $\lambda_{0}(\omega)\leq\lambda(\theta,\lambda_{1},\sigma)<0$ where $\lambda(\theta,\lambda_{1},\sigma)$ will be given explicitly below.
		\end{itemize}
\end{theorem}
\begin{proof}
The proof of the first part is similar to Theorem \ref{the.1}. And we adopt the same notations. We prove the second  part directly. The perturbation in this theorem is a vector which is defined as follows:
$$\textbf{Q}(t)=\begin{pmatrix}
0\\
	\textbf{q}_{2}(t)
\end{pmatrix},
$$ where $\textbf{q}_{2}$ is similar with \eqref{equ.54}. Given a constant $c>0$, for almost all $\omega$, there exists a $k(\omega)$ such that $\vert\frac{W(t,\omega)}{t}\vert<c$ when $t>t_{2k(\omega)-1}$. We omit $\omega$ in the proof to simplify the notation.
Assume that the norm of   solution $X(t)$ at time $t_{2k-1}$ and the angle $\gamma_{2k-1}:=\angle \left\{X(t_{2k-1}), Ox_{2}\right\}$ satisfy the conditions
	\begin{equation}\label{equ.47}
		\begin{aligned}
    &e^{\xi_{1}t_{2k-1}}e^{\beta W(t_{2k-1})}\leq\Vert X(t_{2k-1})\Vert \leq e^{\xi_{2}t_{2k-1}}e^{\beta W(t_{2k-1})},\\
	&\tan\gamma_{2k-1}=e^{-\eta t_{2k-1}}, \quad \eta=-\theta\sigma+(\theta-1)(\alpha_{1}+\alpha_{2})>0.
		\end{aligned}
	\end{equation}

 The perturbed equation \eqref{equ.50}  can be expressed on the interval $[t_{2k-1},t_{2k})$ by two equations
$$\d X_{1}=\alpha_{1}X_{1}\d t+\beta X_{1}\d W,\quad \d X_{2}=[-\alpha_{2}X_{2}+\mathbf{q}_{2}(t)]\d t+\beta X_{2}\d W.$$
	After integrating the equations, we obtain the following expressions:
$$
		\begin{aligned}
		X_{1}(t)&=X_{1}(t_{2k-1})e^{(\alpha_{1}-\frac{1}{2}\beta^{2})(t-t_{2k-1})+\beta[W(t)-W(t_{2k-1})]}   ,\quad t\in [t_{2k-1},t_{2k}),\\
		X_{2}(t)&=X_{2}(t_{2k-1})e^{(-\alpha_{2}-\frac{1}{2}\beta^{2})(t-t_{2k-1})+\beta[W(t)-W(t_{2k-1})]}   ,\quad t\in [t_{2k-1},\tau_{1}],\\
	X_{2}(t)&=e^{(-\alpha_{2}-\frac{1}{2}\beta^{2})(t-\tau_{1})+\beta[W(t)-W(\tau_{1})]} \\
		\cdot\big[	&X_{2}(\tau_{1})+\int_{\tau_{1}}^{t} \mathbf{q}_{2}(\tau)e^{(\alpha_{2}+\frac{1}{2}\beta^{2})(\tau-\tau_{1})-\beta[W(\tau )-W(\tau_{1})]}\d \tau \big] ,\quad t\in [\tau_{1},t_{2k}).
		\end{aligned}$$
	\par If $d_{2k}=0$, we have
\begin{equation}\label{equ.45}
		\begin{aligned}
			\tan \gamma_{2k}=\frac{X_{2}(t_{2k})}{X_{1}(t_{2k})}=\cot \gamma_{2k-1}e^{-(\alpha_{1}+\alpha_{2})(t_{2k}-t_{2k-1})}=e^{-\sigma t_{2k}}.
		\end{aligned}
\end{equation}

By the similar method, for $d_{2k}=-d<0$, we have
\begin{equation}\label{equ.41}
\begin{aligned}
\tan \gamma_{2k}=&\frac{X_{2}(t_{2k})}{X_{1}(t_{2k})}\\
=&\frac{X_{2}(\tau_{1})+\int_{\tau_{1}}^{t_{2k}} \mathbf{q}_{2}(\tau)e^{(\alpha_{2}+\frac{1}{2}\beta^{2})(\tau-\tau_{1})-\beta[W(\tau)-W(\tau_{1})]}\d \tau }{X_{1}(\tau_{1})} \\
&\cdot e^{(\alpha_{2}+\frac{1}{2}\beta^{2})(t_{2k}-\tau_{1})-\beta[W(t_{2k})-W(\tau_{1})]} \\
&\cdot e^{(\alpha_{1}-\frac{1}{2}\beta^{2})(t_{2k}-\tau_{1})+\beta[W(t_{2k})-W(\tau_{1})]}.
\end{aligned}
\end{equation}
So we only consider whether the first  part of   formula \eqref{equ.41} will vanish at time $t=t_{2k}$. We can suppose that $0<X_{1}(\tau_{1})<1$.
Then
\begin{equation}\label{equ.8}
		\begin{aligned}
			&\frac{X_{2}(\tau_{1})+\int_{\tau_{1}}^{t_{2k}} \mathbf{q}_{2}(\tau)e^{(\alpha_{2}+\frac{1}{2}\beta^{2})(\tau-\tau_{1})-\beta[W(\tau)-W(\tau_{1})]}\d \tau }{X_{1}(\tau_{1})}\\
			\leq&
			 \frac{X_{2}(\tau_{1})}{X_{1}(\tau_{1})}-d\int_{\tau_{1}}^{t_{2k}}e^{-\sigma \tau+(\alpha_{2}+\frac{1}{2}\beta^{2})(\tau-\tau_{1})} \cdot e^{-\beta[W(\tau)-W(\tau_{1})]}\d \tau +C_{1}\\
			\leq& \cot \gamma_{2k-1}e^{-(\alpha_{1}+\alpha_{2})(\tau_{1}-t_{2k-1})}-d\int_{\tau_{1}}^{t_{2k}}e^{-\sigma \tau+(\alpha_{2}+\frac{1}{2}\beta^{2})(\tau-\tau_{1})} \cdot e^{-\beta[W(\tau)-W(\tau_{1})]}\d \tau+C_{1}\\
			\leq& e^{-\sigma t_{2k}}e^{(\alpha_{1}+\alpha_{2})(t_{2k}-\tau_{1})}-d\int_{\tau_{1}}^{t_{2k}}e^{-\sigma \tau+(\alpha_{2}+\frac{1}{2}\beta^{2})(\tau-\tau_{1})} \cdot e^{-c \tau}\d \tau+C_{1}\\
			\leq& e^{-\sigma t_{2k}}e^{(\alpha_{1}+\alpha_{2})(t_{2k}-\tau_{1})}-de^{-\sigma\tau_{1}-c\tau_{1}}+C_{1}\\
		\leq& 0,\quad \\
		\end{aligned}
\end{equation}
for $d=[C+C_{1}e^{\sigma\tau_{1}}]e^{c \tau_{1}}$ where $C$ is a constant and  $C_{1}$ can be expressed by the integral
	\begin{equation}\label{equ.63}
\big(\int_{\tau_{1}}^{\tau_{2}}+\int_{\tau_{3}}^{\tau_{4}}\big) [de^{-\sigma\tau}+\mathbf{q}_{2}(\tau)]e^{(\alpha_{2}+\frac{1}{2}\beta^{2})(\tau-\tau_{1})-\beta[W(\tau)-W(\tau_{1})]}\d \tau,
	\end{equation}
	 but it has no effect on exponentially decaying property of $\mathbf{q}_{2}$. The third inequality in \eqref{equ.8} holds by  the estimate $\sigma+\beta<\alpha_{1}+\alpha_{2}$ and $\lim\limits_{t\rightarrow\infty}\frac{W(t)}{t}=0$  $a.s.$  Therefore, for almost all $\omega$, by the continuous dependence of the tangent of the angle $\gamma_{2k}$ on  $d_{2k}$ and  estimates \eqref{equ.45} and \eqref{equ.8}, there exists a value $d_{2k}(\omega)\leq d$ such that $\tan\gamma_{2k}=e^{-\eta t_{2k}}$ holds. We can find that  the rate of $\textbf{Q}(t,\omega)$ is  $c-\sigma$. Let $c$  be small enough such that  $\textbf{Q}(t,\omega)$ still exponentially decays for almost all $\omega$. It is not difficult to find that we can choose an appropriate $\hat{d}_{2k}(\omega)$ to let this angle vanish on $t_{2k}$ which will be used in the proof of the $n$-dimensional situation.

Finally, it is not difficult to show that the norm of the solution $X(t)$ of the perturbed equation \eqref{equ.50} meets  requirements \eqref{equ.47} on the interval $[t_{2k-1},t_{2k})$.
By induction we extend these construction of the equation and the perturbation to the interval $[t_{0},+\infty)$. And by polishing the coefficients, we get the $C^{\infty}$ equations. The proof is completed.
\end{proof}
\begin{remark}\rm
The uniqueness of the solution with negative Lyapunov exponent of perturbed equation \eqref{equ.50} cannot be proved by the same method above, because the condition of Liouville's theorem  \cite{Vrko1978Liouville} is not met.
\end{remark}
\begin{remark}\rm
The results of Theorem \ref{the.3} and Theorem \ref{the.7} are also true for $n$--dimensional situation.  The proof of these results is similar with that of Theorem \ref{the.3} and Theorem \ref{the.7}. So there is no more details.
\end{remark}
\section{Perturbations at diffusion term}
The conclusion not only holds  for the two kinds of perturbations above but also holds for other two kinds of perturbations at the diffusion term. The original equation is still equation \eqref{equ.1}. The first perturbation is expressed by the perturbed equation
\begin{equation}\label{equ.27}
	\d X=A(t)X\d t+[F(t)+Q(t)]X\d W
\end{equation}
 and the second perturbation is expressed by the following equation
\begin{equation}\label{equ.46}
	\d X=A(t)X\d t+[F(t)X+Q(t)]\d W,
\end{equation}
where $W$ is a one dimensional Brownian motion, and $A$ and $F$ are bounded and $C^{\infty}$ functions.
\subsection{The perturbation at homogeneous term}
The perturbed equation \eqref{equ.27} is similar to equation \eqref{equ.2}. Accordingly, the conclusion can also be established with respect to it.
 \begin{theorem}\label{the.5}
 	\begin{itemize}
      \item [(1)] For any constants $\lambda_{2}\geq\lambda_{1}>0$, $\theta>1$, there exists a $2$-dimensional equation \eqref{equ.1} with bounded $C^{\infty}$ coefficients and Lyapunov exponents $\lambda_{1}, \lambda_{2}$.
      \par \item [(2)] For this equation and any constant $\sigma\in(0,\frac{(\theta-1)\lambda_{1}}{\theta})$, there exists a $C^{\infty}$ random  perturbation $\textbf{Q}$ satisfying that for almost all $\omega$,  $\textbf{Q}(t,\omega)$ is exponentially decaying, such that the perturbed equation \eqref{equ.27} has a unique solution $X_0$  with negative Lyapunov exponent $\lambda_{0}(\omega)\leq\lambda(\theta,\lambda_{1},\sigma)<0$ where $\lambda(\theta,\lambda_{1},\sigma)$ will be given explicitly below;
       the Lyapunov exponent of the other solution $X_{1}$ which is linearly independent of $X_{0}$ is positive and is greater than $\vert\lambda_{0}(\omega)\vert$.
      \end{itemize}
\end{theorem}
\begin{proof}
 The existence of equation \eqref{equ.1} with Lyapunov exponents $\lambda_{1}$ and $\lambda_{2}$ can be proved as same as Theorem \ref{the.1}.
\par  Next, we will prove the second parts of the theorem by using the induction method. We still assume that the norm of the solution $X(t)$ at time $t=t_{2k-1}$ and the angle $\gamma_{2k-1}:=\angle \left\{X(t_{2k-1}), Ox_{2}\right\}$ satisfy the conditions
$$	e^{\xi_{1}t_{2k-1}}e^{W(t_{2k-1})}\leq\Vert X(t_{2k-1})\Vert \leq e^{\xi_{2}t_{2k-1}}e^{W(t_{2k-1})},$$
$$\tan\gamma_{2k-1}=e^{-\eta t_{2k-1}}, \quad \eta=-\theta\sigma+(\theta-1)(\alpha_{1}+\alpha_{2})>0,$$
which are the same as above. But the perturbation $\textbf{Q}(t,\omega)$ is different with the definition above.  According to limit $\lim\limits_{t\rightarrow\infty}\frac{W(t)}{t}=0 ~a.s.$, we know that for almost all $\omega$, there exists a $k(\omega)$ such that $-c\beta<\frac{W(t,\omega)}{t}<c\beta$ when $t\geq t_{2k(\omega)-1}$, where $0<c<\frac{1}{2}$ is a constant.  Denote $\hat{c}:=(1-\frac{(1-2c)}{1+c})t_{2k}$. Then the perturbation can be expressed as follows: 
$$q_{11}(t)=q_{12}(t)=q_{22}(t)=0,\quad t\in[t_{2k-1},t_{2k}),$$
\par
\begin{equation}\label{equ.57}
\textbf{q}_{21}(t)=d_{2k}e^{-\zeta (t)}\left\{
\begin{aligned}
	&0,   &t\in[t_{2k-1}, \tau_{1}),\\
	&e_{01}(t, \tau_{1},\tau_{2}),  &t\in[\tau_{1}, \tau_{2}),\\
	&1,   &t\in[\tau_{2}, \tau_{3}),\\
	&e_{10}(t, \tau_{3},\tau_{4}),  &t\in[\tau_{3}, \tau_{4}),\\
\end{aligned}
\right.
\end{equation}	
$$\epsilon(t)=e^{-t^{2}},\quad \tau_{1}=t_{2k}-\hat{c}-2\epsilon(t_{2k}), \quad \tau_{2}=\tau_{1}+\epsilon(t_{2k}),  $$
$$\tau_{3}=\tau_{2}+\hat{c}, \quad  \tau_{4}=\tau_{3}+\epsilon(t_{2k})=t_{2k},$$
where $d_{2k}$ can be given explicitly below and $\zeta(t)=(\alpha_{1}+\alpha_{2})t$.  We change the length of $[\tau_{2},\tau_{3}]$ from 1 to $\hat{c}$.
\par The perturbed equation \eqref{equ.27} has the expression on the interval $[t_{2k-1},t_{2k})$ as follows:
$$\d X_{1}=\alpha_{1}X_{1}\d t+\beta X_{1}\d W,\quad \d X_{2}=-\alpha_{2}X_{2}\d t+[\textbf{q}_{21}(t)X_{1}+\beta X_{2}]\d W.$$
Then we get the following solution on the interval $[t_{2k-1},t_{2k}]$
\begin{align*}
X_{1}(t)&=X_{1}(t_{2k-1})e^{(\alpha_{1}-\frac{1}{2}\beta^{2})(t-t_{2k-1})+\beta[W(t)-Wt_{2k-1})]},\quad t\in [t_{2k-1},t_{2k}),\\
X_{2}(t)&=X_{2}(t_{2k-1})e^{(-\alpha_{2}-\frac{1}{2}\beta^{2})(t-t_{2k-1})+\beta[W(t)-W(t_{2k-1})]},\quad t\in [t_{2k-1},\tau_{1}],\\
X_{2}(t)&=\Phi(t)\big[X_{2}(\tau_{1})-\int_{\tau_{1}}^{t} \beta \textbf{q}_{21}(\tau) X_{1}(\tau)\Phi^{-1}(\tau)\d \tau\\
&\qquad\qquad\qquad+\int_{\tau_{1}}^{t}\textbf{q}_{21}(\tau)X_{1}(\tau) \Phi^{-1}(\tau)\d W(\tau) \big], \qquad t\in [\tau_{1},t_{2k}),
\end{align*}
where $\Phi(t)$ has the expression
$$\Phi(t)=e^{(-\alpha_{2}-\frac{1}{2}\beta^{2})(t-\tau_{1})+\beta[W(t)-W(\tau_{1})]}.$$
\par Next, we will estimate  the range of $\tan \gamma_{2k}$ by giving different values to $d_{2k}$. If $d_{2k}=0$, we have the following estimate
$$
	\begin{aligned}
		\tan \gamma_{2k}&=\frac{X_{2}(t_{2k})}{X_{1}(t_{2k})}
		=\cot \gamma_{2k-1}e^{-(\alpha_{1}+\alpha_{2})(t_{2k}-t_{2k-1})}
		=e^{-\sigma t_{2k}}.
	\end{aligned}$$
If $d_{2k}=d >0$,  the upper bound of $\tan \gamma_{2k}$ is obtained as follows:
$$
	\begin{aligned}
		\tan \gamma_{2k}=&\frac{X_{2}(t_{2k})}{X_{1}(t_{2k})}\\
		=&\frac{X_{2}(\tau_{1})-\int_{\tau_{1}}^{t_{2k}} \beta \textbf{q}_{21}(\tau)X_{1}(\tau) \Phi^{-1}(\tau)\d \tau+\int_{\tau_{1}}^{t_{2k}}\textbf{q}_{21}(\tau) X_{1}(\tau)\Phi^{-1}(\tau)\d W(\tau) }{X_{1}(\tau_{1})} \\
		&\cdot e^{(-\alpha_{2}-\frac{1}{2}\beta^{2})(t_{2k}-t_{2k-1})+\beta[W(t_{2k})-W(t_{2k-1})]} \\ &\cdot e^{-(\alpha_{1}-\frac{1}{2}\beta^{2})(t_{2k}-t_{2k-1})-\beta[W(t_{2k})-W(t_{2k-1})]}.
	\end{aligned}$$
So we only consider the first component of the equality,
$$
	\begin{aligned}
		&\frac{X_{2}(\tau_{1})-\int_{\tau_{1}}^{t_{2k}} \beta \textbf{q}_{21}(\tau)X_{1}(\tau) \Phi^{-1}(\tau)\d \tau+\int_{\tau_{1}}^{t_{2k}}\textbf{q}_{21}(\tau) X_{1}(\tau)\Phi^{-1}(\tau)\d W(\tau) }{X_{1}(\tau_{1})}\\
			\end{aligned}$$
$$
	\begin{aligned}
		\leq& \frac{X_{2}(\tau_{1})}{X_{1}(\tau_{1})}-e^{-(\alpha_{1}+\alpha_{2})\tau_{1}}\big[\int_{\tau_{1}}^{t_{2k}} d\beta \d \tau-\int_{\tau_{1}}^{t_{2k}}d \d W(\tau)\big] +C_{1}\\
		\leq & e^{-\sigma t_{2k}+(\alpha_{1}+\alpha_{2})(t_{2k}-\tau_{1})}-e^{-(\alpha_{1}+\alpha_{2})\tau_{1}}\big[d\beta(t_{2k}-\tau_{1})-d(W(t_{2k})-W(\tau_{1}))\big]+C_{1}\\
		\leq & e^{-\sigma t_{2k}+(\alpha_{1}+\alpha_{2})(t_{2k}-\tau_{1})}-e^{-(\alpha_{1}+\alpha_{2})\tau_{1}}[dc\beta t_{2k}]+C_{1},
	\end{aligned}$$
where $C_{1}$ has a similar expression to \eqref{equ.63} and has no effect on exponentially decaying property of $\mathbf{q}_{21}$.
We should remark that the stochastic integral $\int_{\tau_{1}}^{\tau_{2}}\textbf{q}_{21}(\tau)\d W$ has the following expression by It\^o's formula:
$$\int_{\tau_{1}}^{\tau_{2}}\textbf{q}_{21}(\tau)\d W=\textbf{q}_{21}(\tau)W(\tau)\Big|^{\tau_{2}}_{\tau_{1}}-\int_{\tau_{1}}^{\tau_{2}}\textbf{q}'_{21}(\tau)W(\tau)\d \tau,$$
where $\textbf{q}'$ is the derivative of $\textbf{q}$ with respect to $t$. It
 can be controlled by inequality $-c\beta<\frac{W(t)}{t}<c\beta$. It's not difficult to proof that  $C_{1}$ has no effect on exponentially decaying property of $\mathbf{q}_{21}$. According to the inequality above,  $\tan \gamma_{2k}$ will vanish by the value
$$d= \frac{1}{c\beta t_{2k}}[e^{-\sigma t_{2k}+(\alpha_{1}+\alpha_{2})t_{2k}}+C_{1}],$$

Because $\tan \gamma_{2k}$ depends on $d$ continuously, there exists a $d_{2k}(\omega)\leq d$ such that the angle satisfies $\tan \gamma_{2k}=e^{-\eta t_{2k}}$. It implies that for almost all $\omega$, the perturbation $\textbf{Q}(t,\omega)$ is exponentially decaying:
$$\Vert Q (t,\omega)\Vert\leq d(t,\omega)e^{-\zeta(t)}\leq Ce^{-(\sigma+\epsilon)t}, $$
where $\epsilon$ is arbitrarily small and $C$ is a constant. The next step of the proof is same  as that in the proof of Theorem \ref{the.1}. We can get the estimate
$$	e^{\xi_{1}t_{2k}}e^{\beta W(t_{2k})}\leq\Vert X(t_{2k})\Vert \leq e^{\xi_{2}t_{2k}}e^{\beta W(t_{2k})}.$$

The remaining proof is similar to that of Theorem \ref{the.1}, so we will not repeat here. In addition, the perturbed equation satisfies the Lyapunov inequality \cite{1998Stochastic} and Liouville's theorem \cite{Vrko1978Liouville}, so the uniqueness of the solution with negative Lyapunov exponent is guaranteed.
\end{proof}

\subsection{The perturbation at inhomogeneous term}
 We try now to solve another similar problem.  We have the following theorem for the the perturbed equation \eqref{equ.46}.
 
\begin{theorem}\label{the.6}
 	\begin{itemize}
	 \item [(1)] For any constants $\lambda_{2}\geq\lambda_{1}>0$, $\theta>1$, there exists a $2$-dimensional equation \eqref{equ.1} with bounded $C^{\infty}$ coefficients and Lyapunov exponents $\lambda_{1}, \lambda_{2}$.
      \item [(2)] For this equation and any constant  $\sigma\in(0,\frac{(\theta-1)\lambda_{1}}{\theta})$, there exists a $C^{\infty}$ random perturbation $\textbf{Q}$ satisfying that for almost all $\omega$, $\textbf{Q}(t,\omega)$ is exponentially decaying, such that the perturbed equation \eqref{equ.46} has a solution with negative Lyapunov exponent $\lambda_{0}(\omega)\leq\lambda(\theta,\lambda_{1},\sigma)<0$ where $\lambda(\theta,\lambda_{1},\sigma)$ will be given explicitly below.
	 	\end{itemize}
\end{theorem}
 \begin{proof}
The proof of the first part is similar to Theorem \ref{the.1}, so we mainly discuss different parts.
\par  In the second part, we use same assumptions
$$e^{\xi_{1}t_{2k-1}}e^{W(t_{2k-1})}\leq\Vert X(t_{2k-1})\Vert \leq e^{\xi_{2}t_{2k-1}}e^{W(t_{2k-1})},$$
$$\tan\gamma_{2k-1}=e^{-\eta t_{2k-1}}, \quad \eta=-\theta\sigma+(\theta-1)(\alpha_{1}+\alpha_{2})>0.$$
\par The perturbation $\textbf{Q}(t,\omega)$ is expressed by vector
$$\textbf{Q}(t,\omega)=\left(\begin{array}{c}
	0 \\
	\textbf{q}_{2}(t,\omega)
\end{array}\right),$$
where $\textbf{q}_{2}$ has similar form with \eqref{equ.57} but the  value of $\zeta(t)$  is changed from $(\alpha_{1}+\alpha_{2})t$ to   $(\alpha_{2}+\frac{1}{2}\beta^{2})t-\beta W(t)$. And the components of the solution $X(t)=(X_{1}(t), X_{2}(t))$ 
are expressed as follows:
$$
	\begin{aligned}
	X_{1}(t)&=X_{1}(t_{2k-1})e^{(\alpha_{1}-\frac{1}{2}\beta^{2})(t-t_{2k-1})+\beta[W(t)-Wt_{2k-1})]}   ,\quad t\in [t_{2k-1},t_{2k}),\\
	X_{2}(t)&=X_{2}(t_{2k-1})e^{(-\alpha_{2}-\frac{1}{2}\beta^{2})(t-t_{2k-1})+\beta[W(t)-W(t_{2k-1})]}   ,\quad t\in [t_{2k-1},\tau_{1}],
		\end{aligned}$$
$$
	\begin{aligned}	X_{2}(t)=\Phi(t)
	\big 	[X_{2}(\tau_{1})-\int_{\tau_{1}}^{t} \beta \textbf{q}_{2}(\tau) \Phi^{-1}(\tau)\d \tau+\int_{\tau_{1}}^{t}\textbf{q}_{2}(\tau) \Phi^{-1}(\tau)\d W(\tau) \big ] ,\quad t\in [\tau_{1},t_{2k}),
	\end{aligned}$$
where $\Phi(t)$ has the expression
$$\Phi(t)=e^{(-\alpha_{2}-\frac{1}{2}\beta^{2})(t-\tau_{1})+\beta[W(t)-W(\tau_{1})]} .$$
\par The value of $\tan \gamma_{2k}$ has the same lower bound:
$$
\tan \gamma_{2k}=\frac{X_{2}(t_{2k})}{X_{1}(t_{2k})}
=\cot \gamma_{2k-1}e^{-(\alpha_{1}+\alpha_{2})(t_{2k}-t_{2k-1})}
=e^{-\sigma t_{2k}}.
$$
The difference appears in the estimation of upper bound. We take similar steps:
\begin{align*}
\tan \gamma_{2k}
=&\frac{X_{2}(t_{2k})}{X_{1}(t_{2k})}\\
=&\frac{X_{2}(\tau_{1})-\int_{\tau_{1}}^{t_{2k}} \beta \textbf{q}_{2}(\tau) \Phi^{-1}(\tau)\d \tau+\int_{\tau_{1}}^{t_{2k}}\textbf{q}_{2}(\tau) \Phi^{-1}(\tau)\d W(\tau) }{X_{1}(\tau_{1})}\\
&\cdot e^{(-\alpha_{2}-\frac{1}{2}\beta^{2})(t_{2k}-\tau_{1})+\beta[W(t_{2k})-W(\tau_{1})]} \\ &\cdot e^{-(\alpha_{1}-\frac{1}{2}\beta^{2})(t_{2k}-\tau_{1})-\beta[W(t_{2k})-W(\tau_{1})]},
\end{align*}
and it is easy to see that the first component of the inequality satisfies the following estimates
\begin{align*}
&\frac{X_{2}(\tau_{1})-\int_{\tau_{1}}^{t_{2k}} \beta \textbf{q}_{2}(\tau) \Phi^{-1}(\tau)\d \tau+\int_{\tau_{1}}^{t_{2k}}\textbf{q}_{2}(\tau) \Phi^{-1}(\tau)\d W(\tau) }{X_{1}(\tau_{1})}\\
\leq &\frac{X_{2}(\tau_{1})}{X_{1}(\tau_{1})}-e^{-(\alpha_{2}+\frac{1}{2}\beta^{2})\tau_{1}+\beta W(\tau_{1})}\big [\int_{\tau_{1}}^{t_{2k}} d\beta \d \tau+\int_{\tau_{1}}^{t_{2k}}d \d W(\tau)\big ]+C_{1}\\
\leq & e^{-\sigma t_{2k}+(\alpha_{1}+\alpha_{2})(t_{2k}-\tau_{1})}-e^{-(\alpha_{2}+\frac{1}{2}\beta^{2})\tau_{1}+\beta W(\tau_{1})}\big [d\beta(t_{2k}-\tau_{1})-d(W(t_{2k})-W(\tau_{1}))\big]+C_{1}\\
\leq & e^{-\sigma t_{2k}+(\alpha_{1}+\alpha_{2})(t_{2k}-\tau_{1})}-e^{-(\alpha_{2}+\frac{1}{2}\beta^{2})\tau_{1}+\beta W(\tau_{1})}\big [dc\beta t_{2k}\big ]+C_{1},
\end{align*}
where $C_{1}$ has a similar expression to \eqref{equ.63} and it has no effect on exponentially decaying property of $\mathbf{q}_{2}$.
Then the angle will have the estimate $\tan \gamma_{2k}\leq 0$ by taking
$$d= \frac{1}{c\beta t_{2k}}e^{(\alpha_{2}+\frac{1}{2}\beta^{2})\tau_{1}-\beta W(\tau_{1})}[e^{-\sigma t_{2k}+(\alpha_{1}+\alpha_{2})(t_{2k}-\tau_{1})}+C_{1}].$$
Because the angle depends on $d_{2k}$ continuously, for almost all $\omega$, we can take an appropriate $d_{2k}(\omega)\leq d$ such that the angle satisfies $\tan \gamma_{2k}=e^{-\eta t_{2k}}$. So we have constructed the desired perturbation. It is not difficult to get the estimate
$$e^{\xi_{1}t_{2k}}e^{\beta W(t_{2k})}\leq\Vert X(t_{2k})\Vert \leq e^{\xi_{2}t_{2k}}e^{\beta W(t_{2k})}.$$
\par Finally, we prove the conclusion by the similar discussion as in Theorem \ref{the.1}. 
 \end{proof}
\begin{remark}\rm
The results of Theorem \ref{the.5} and Theorem \ref{the.6} are also true for $n$--dimensional situation. 
\end{remark}

\section{Discussions}
On the one hand, we cannot get a deterministic perturbation in the last three cases. As we know, it needs two procedures to get a deterministic perturbation. They are initial condition at $t_{2k-1}$ and inductive hypothesis. First, assume that we get a deterministic perturbation on the interval $[t_{2k},\infty)$ such that the perturbed equation satisfies the inductive hypothesis
\begin{equation}\label{equ.48}
	\begin{aligned}
	&e^{\xi_{1}t_{2k-1}}e^{\beta W(t_{2k-1})}\leq\Vert X(t_{2k-1})\Vert \leq e^{\xi_{2}t_{2k-1}}e^{\beta W(t_{2k-1})},\\
e^{-\eta t_{2k-1}}&\leq\tan\gamma_{2k-1}\leq e^{-\sigma t_{2k-1}}, \quad \eta=-\theta\sigma+(\theta-1)(\alpha_{1}+\alpha_{2})>0.
\end{aligned}
\end{equation}
 Then on the interval $[t_{1},t_{2k-1}]$, we need a determined perturbation such that the perturbed equation satisfies the initial condition: for  almost all $\omega$, there exists a large enough $k(\omega)$, such that for a fixed $k>k(\omega)$, we have estimate \eqref{equ.48} at time $t=t_{2k-1}$. Indeed, this situation is difficult to achieve under a deterministic perturbation. Let us consider the $2$-dimensional perturbed equation \eqref{equ.50}. The probability of solution $X(t)$ which arrives at second or forth quadrants at time $t_{2k-1}$ is positive. We define this set of positive measures by $\Omega_{1}$.  From the form of the solution of the  perturbed equation
\begin{align*}	
X_{1}(t)&=X_{1}(t_{2k-1})e^{(\alpha_{1}-\frac{1}{2}\beta^{2})(t-t_{2k-1})+\beta[W(t)-W(t_{2k-1})]}   ,\quad t\in [t_{2k-1},t_{2k}],\\
X_{2}(t)&=X_{2}(t_{2k-1})e^{(-\alpha_{2}-\frac{1}{2}\beta^{2})(t-t_{2k-1})+\beta[W(t)-W(t_{2k-1})]}   ,\quad t\in [t_{2k-1},\tau_{1}],\\
X_{2}(t)&=e^{(-\alpha_{2}-\frac{1}{2}\beta^{2})(t-\tau_{1})+\beta[W(t)-W(\tau_{1})]}  \cdot \big[X_{2}(\tau_{1})\\
&\qquad+\int_{\tau_{1}}^{t} q_{21}(\tau)e^{(\alpha_{2}+\frac{1}{2}\beta^{2})(\tau-\tau_{1})-\beta[W(\tau)-W(\tau_{1})]}\d \tau \big] ,\quad t\in [\tau_{1},t_{2k}],
\end{align*}
we notice that for any $\omega\in\Omega_{1}$,  solution $X(t,\omega)$ will never leave the second or forth quadrants. And  solution $X(t,\omega)$ is unstable because the absolute value of components $X_{1}(t,\omega)$ and $X_{2}(t,\omega)$ of solution $X(t,\omega)$ are always increasing when $t\rightarrow\infty$. So deterministic perturbations cannot have the similar conclusion.
 \par On the other hand, it is difficult to get an equation \eqref{equ.1} and a perturbation such that Lyapunov exponents of the perturbed equation \eqref{equ.2} are changed from all positive to all negative.
  The uniqueness of solution with negative Lyapunov exponent partly reflects this situation. There is always an exponentially increasing solution under the perturbation above. We tried many other methods but the  perturbation is always unable to make all solutions stable.

\section*{Acknowledgements}
The authors are very grateful to the referee for his/her very careful reading of the paper and for many valuable suggestions which lead to significant improvements of the paper. This work is partially supported by NSFC Grants 11871132, 11925102, Dalian High-level
Talent Innovation Project (Grant 2020RD09), and Xinghai Jieqing fund from Dalian University
of Technology.

\end{document}